\documentclass[final,leqno]{siamltex704}
\usepackage{epsfig}
\usepackage{amsmath}
\usepackage{amssymb}
\usepackage{tikz}
\usepackage{graphicx}
\usepackage[notcite,notref]{showkeys}
\newtheorem{algorithm}{Weak Galerkin Algorithm}

\newcommand{\bn}{{\bf n}}

\def\T{{\mathcal T}}
\def\E{{\mathcal E}}

\def\l{{\langle}}
\def\r{{\rangle}}

\def\bn{{\bf n}}

\def\bbQ{\mathbb{Q}}
\newcommand{\pT}{{\partial T}}
\newtheorem{defi}{Definition}[section]
\def\3bar{{|\hspace{-.02in}|\hspace{-.02in}|}}

\title{A Weak Galerkin Finite Element Method with Polynomial Reduction}

\author{Lin Mu\thanks{Department of Mathematics, Michigan State University,
      East Lansing, MI 48824 (linmu@ msu.edu)} \and Junping
Wang\thanks{Division of Mathematical Sciences, National Science
Foundation, Arlington, VA 22230 (jwang@\break nsf.gov). The research
of Wang was supported by the NSF IR/D program, while working at
National Science Foundation. However, any opinion, finding, and
conclusions or recommendations expressed in this material are those
of the author and do not necessarily reflect the views of the
National Science Foundation,} \and Xiu Ye\thanks{Department of
Mathematics, University of Arkansas at Little Rock, Little Rock, AR
72204 (xxye@ualr.edu). This research was supported in part by
National Science Foundation Grant DMS-1115097.}}

\begin{document}

\maketitle

\begin{abstract}
The novel idea of weak Galerkin (WG) finite element methods is on
the use of weak functions and their weak derivatives defined as
distributions. Weak functions and weak derivatives can be
approximated by polynomials with various degrees. Different
combination of polynomial spaces leads to different weak Galerkin
finite element methods, which makes WG methods highly flexible and
efficient in practical computation. This paper explores the
possibility of optimal combination of polynomial spaces that
minimize the number of unknowns in the numerical scheme, yet without
compromising the accuracy of the numerical approximation. For
illustrative purpose, the authors use second order elliptic problems
to demonstrate the basic idea of polynomial reduction. A new weak
Galerkin finite element method is proposed and analyzed. This new
finite element scheme features piecewise polynomials of degree $k\ge
1$ on each element plus piecewise polynomials of degree $k-1\ge 0$
on the edge or face of each element. Error estimates of optimal
order are established for the corresponding WG approximations in
both a discrete $H^1$ norm and the standard $L^2$ norm. In addition,
the paper presents a great deal of numerical experiments to
demonstrate the power of the WG method in dealing with finite
element partitions consisting of arbitrary polygons in two
dimensional spaces or polyhedra in three dimensional spaces. The
numerical examples include various finite element partitions such as
triangular mesh, quadrilateral mesh, honey comb mesh in 2d and mesh
with deformed cubes in 3d. The numerical results show a great
promise of the robustness, reliability, flexibility and accuracy of
the WG method.
\end{abstract}


\pagestyle{myheadings}

\section{Introduction}\label{Section:Introduction}

This paper is concerned with weak Galerkin (WG) finite element
methods by exploring optimal use of polynomial approximating spaces.
In general, weak Galerkin refers to finite element techniques for
partial differential equations in which differential operators
(e.g., gradient, divergence, curl, Laplacian) are approximated by
weak forms as distributions. The main idea of weak Galerkin finite
element methods is the use of weak functions and their corresponding
discrete weak derivatives in algorithm design. For the second order
elliptic equation, weak functions have the form of $v=\{v_0,v_b\}$
with $v=v_0$ inside of each element and $v=v_b$ on the boundary of
the element. Both $v_0$ and $v_b$ can be approximated by polynomials
in $P_\ell(T)$ and $P_s(e)$ respectively, where $T$ stands for an
element and $e$ the edge or face of $T$, $\ell$ and $s$ are
non-negative integers with possibly different values. Weak
derivatives are defined for weak functions in the sense of
distributions. For computing purpose, one needs to approximate the
weak derivatives by polynomials. For example, for the weak gradient
operator, one may approximate it in the polynomial space
$[P_m(T)]^d$. Various combination of $(P_\ell(T),P_s(e),[P_m(T)]^d)$
leads to different class of weak Galerkin methods tailored for
specific partial differential equations. The goal of this paper is
to explore optimal combination of the polynomial spaces $P_\ell(T)$
and $P_s(e)$ that minimizes the number of unknowns without
compromising the rate of convergence for the corresponding WG
method.

For simplicity, we demonstrate the idea of optimality for
polynomials by using the second order elliptic problem that seeks an
unknown function $u$ satisfying
\begin{eqnarray}
-\nabla\cdot (a\nabla u)&=&f,\quad \mbox{in}\;\Omega,\label{pde}\\
u&=&g,\quad\mbox{on}\;\partial\Omega,\label{bc}
\end{eqnarray}
where $\Omega$ is a polytopal domain in $\mathbb{R}^d$ (polygonal or
polyhedral domain for $d=2,3$), $\nabla u$ denotes the gradient of
the function $u$, and $a$ is a symmetric
$d\times d$ matrix-valued function in $\Omega$. We shall assume that
there exists a positive number $\lambda>0$ such that
\begin{equation}\label{ellipticity}
\xi^ta\xi\ge \lambda \xi^t\xi,\qquad\forall
\xi\in\mathbb{R}^d.
\end{equation}
Here $\xi$ is understood as a column vector and $\xi^t$ is the transpose
of $\xi$.

A weak Galerkin method has been introduced and analyzed in \cite{wy}
for second order elliptic equations based on a {\em discrete weak
gradient} arising from local {\em RT} \cite{rt} or {\em BDM}
\cite{bdm} elements. More specifically, in the case of {\em BDM}
element of order $k\ge 1$, the gradient space is taken as
$[P_m(T)]^d\equiv [P_k(T)]^d$ and the weak functions are defined by
using $(P_\ell(T),P_s(e))\equiv (P_{k-1}(T), P_k(e))$. For the {\em
RT} element of $k\ge 0$, the gradient space is the usual {\em RT}
element for the vector component while the weak functions are given
by $(P_\ell(T),P_s(e))\equiv (P_{k}(T), P_k(e))$. Due to the use of
the {\em RT} and {\em BDM} elements, the WG finite element
formulation of \cite{wy} is limited to classical finite element
partitions of triangles ($d=2$) or tetrahedra ($d=3$). In addition,
the corresponding WG scheme exhibits a close connection with the
standard mixed finite element method for (\ref{pde})-(\ref{bc}).

The main goal of this paper is to investigate the
possibility of optimal combination of polynomial spaces that
minimize the number of unknowns in the numerical scheme without
compromising the order of
convergence. The new WG scheme will use the configuration of
$(P_k(T), P_{k-1}(e), P_{k-1}(T)^d)$, and the corresponding WG
solution converges to the exact solution of (\ref{pde})-(\ref{bc})
with rate of $O(h^k)$ in $H^1$ and $O(h^{k+1})$ in $L^2$ norm,
provided that the exact solution of the original problem is
sufficiently smooth. It should be pointed out that the unknown $v_0$
associated with the interior of each element can be eliminated in
terms of the unknown $v_b$ defined on the element boundary in
practical implementation. This means that, for problems in
$\mathbb{R}^2$, only edges of the finite element partition shall
contribute unknowns ($k$ unknowns from each edge) to the global
stiffness matrix problem. The new WG scheme is, therefore, a natural
extension of the classical Crouzix-Raviart $P_1$ non-conforming
triangular element to arbitrary order and arbitrary polygonal
partitions.

It have been proved rigorously in \cite{wy-m} that $P_k$ type of
polynomials can be used in weak Galerkin finite element procedures
on any polygonal/polyhedral elements. It contrasts to the use of
polynomials $P_k$ for triangular elements and tensor products $Q_k$
for quadrilateral elements in classic finite element methods. In
practice, allowing arbitrary shape in finite element partition
provides a great flexibility in both numerical approximation and
mesh generation, especially in regions where the domain geometry is
complex. Such a flexibility is also very much appreciated in
adaptive mesh refinement methods. Another objective of this paper is
to study the reliability, flexibility and accuracy of the weak
Galerkin method through extensive numerical tests. The first and
second order weak Galerkin elements are tested on partitions with
different shape of polygons and polyhedra. Our numerical results
show optimal order of convergence for $k=1,2$ on triangular,
quadrilateral, honey comb meshes in 2d and deformed cube in 3d.

One close relative of the WG finite element method of this paper is
the hybridizable discontinuous Galerkin (HDG) method \cite{cgl}. But
these two methods are fundamentally different in concept and
formulation. The HDG method is formulated by using the standard
mixed method approach for the usual system of first order equations,
while the key to WG is the use of discrete weak differential
operators. For the second order elliptic problem
(\ref{pde})-(\ref{bc}), these two methods share the same feature of
approximating first order derivatives or fluxes through a formula
that was commonly employed in the mixed finite element method. For
high order PDEs, such as the biharmonic equation
\cite{mwy-biharmonic}, the WG method is greatly
different from the HDG. It should be emphasized that the concept of
weak derivatives makes WG a widely applicable numerical technique
for a large variety of partial differential equations which we shall
report in forthcoming papers.

The paper is organized as follows. In Section
\ref{Section:weak-gradient}, we shall review the definition of the
weak gradient operator and its discrete analogues. In Section
\ref{Section:wg-fem}, we shall describe a new WG scheme. Section
\ref{Section:wg-massconservation} will be devoted to a discussion of
mass conservation for the WG scheme. In Section
\ref{Section:L2projections}, we shall present some technical
estimates for the usual $L^2$ projection operators. Section
\ref{Section:error-analysis} is used to derive an optimal order
error estimate for the WG approximation in both $H^1$ and $L^2$
norms. Finally in Section \ref{Section:numerical}, we shall present
some numerical results that confirm the theory developed in earlier
sections.

\section{Weak Gradient and Discrete Weak Gradient}\label{Section:weak-gradient}

Let $K$ be any polytopal domain with boundary $\partial K$. A {\em
weak function} on the region $K$ refers to a function $v=\{v_0,
v_b\}$ such that $v_0\in L^2(K)$ and $v_b\in H^{\frac12}(\partial
K)$. The first component $v_0$ can be understood as the value of $v$
in $K$, and the second component $v_b$ represents $v$ on the
boundary of $K$. Note that $v_b$ may not necessarily be related to
the trace of $v_0$ on $\partial K$ should a trace be well-defined.
Denote by $W(K)$ the space of weak functions on $K$; i.e.,
\begin{equation}\label{W(K)}
W(K)= \{v=\{v_0, v_b \}:\ v_0\in L^2(K),\; v_b\in
H^{\frac12}(\partial K)\}.
\end{equation}
Define $(v,w)_D=\int_Dvwdx$ and $\l v, w\r_\gamma=\int_\gamma vwds$.

The weak gradient operator, as was introduced in \cite{wy}, is
defined as follows for the completion of the paper.
\medskip

\begin{defi}
The dual of $L^2(K)$ can be identified with itself by using the
standard $L^2$ inner product as the action of linear functionals.
With a similar interpretation, for any $v\in W(K)$, the {\em weak
gradient} of $v$ is defined as a linear functional $\nabla_w v$ in
the dual space of $H(div,K)$ whose action on each $q\in H(div,K)$ is
given by
\begin{equation}\label{weak-gradient}
(\nabla_w v, q)_K = -(v_0, \nabla\cdot q)_K + \langle v_b,
q\cdot\bn\rangle_{\partial K},
\end{equation}
where $\bn$ is the outward normal direction to $\partial K$,
$(v_0,\nabla\cdot q)_K=\int_K v_0 (\nabla\cdot q)dK$ is the action
of $v_0$ on $\nabla\cdot q$, and $\langle v_b,
q\cdot\bn\rangle_{\partial K}$ is the action of $q\cdot\bn$ on
$v_b\in H^{\frac12}(\partial K)$.
\end{defi}

\medskip

The Sobolev space $H^1(K)$ can be embedded into the space $W(K)$ by
an inclusion map $i_W: \ H^1(K)\to W(K)$ defined as follows
$$
i_W(\phi) = \{\phi|_{K}, \phi|_{\partial K}\},\qquad \phi\in H^1(K).
$$
With the help of the inclusion map $i_W$, the Sobolev space $H^1(K)$
can be viewed as a subspace of $W(K)$ by identifying each $\phi\in
H^1(K)$ with $i_W(\phi)$. Analogously, a weak function
$v=\{v_0,v_b\}\in W(K)$ is said to be in $H^1(K)$ if it can be
identified with a function $\phi\in H^1(K)$ through the above
inclusion map. It is not hard to see that the weak gradient is
identical with the strong gradient (i.e., $\nabla_w v=\nabla v$) for
smooth functions $v\in H^1(K)$.

Denote by $P_{r}(K)$ the set of polynomials on $K$ with degree no
more than $r$. We can define a discrete weak gradient operator by
approximating $\nabla_w$ in a polynomial subspace of the dual of
$H(div,K)$.

\smallskip

\begin{defi}
The discrete weak gradient operator, denoted by
$\nabla_{w,r, K}$, is defined as the unique polynomial
$(\nabla_{w,r, K}v) \in [P_r(K)]^d$ satisfying the following
equation
\begin{equation}\label{dwd}
(\nabla_{w,r, K}v, q)_K = -(v_0,\nabla\cdot q)_K+ \langle v_b,
q\cdot\bn\rangle_{\partial K},\qquad \forall q\in [P_r(K)]^d.
\end{equation}
\end{defi}

By applying the usual integration by part to the first term on the
right hand side of (\ref{dwd}), we can rewrite the equation
(\ref{dwd}) as follows
\begin{equation}\label{dwd-2}
(\nabla_{w,r, K}v, q)_K = (\nabla v_0,q)_K+ \langle v_b-v_0,
q\cdot\bn\rangle_{\partial K},\qquad \forall q\in [P_r(K)]^d.
\end{equation}

\section{Weak Galerkin Finite Element Schemes}\label{Section:wg-fem}

Let ${\cal T}_h$ be a partition of the domain $\Omega$ consisting of
polygons in two dimension or polyhedra in three dimension satisfying
a set of conditions specified in \cite{wy-m}. Denote by ${\cal E}_h$
the set of all edges or flat faces in ${\cal T}_h$, and let ${\cal
E}_h^0={\cal E}_h\backslash\partial\Omega$ be the set of all
interior edges or flat faces. For every element $T\in \T_h$, we
denote by $h_T$ its diameter and mesh size $h=\max_{T\in\T_h} h_T$
for ${\cal T}_h$.

For a given integer $k\ge 1$, let $V_h$ be the weak Galerkin finite
element space associated with $\T_h$ defined as follows
\begin{equation}\label{vhspace}
V_h=\{v=\{v_0,v_b\}:\; v_0|_T\in P_k(T),\ v_b|_e\in P_{k-1}(e),\ e\in\pT,  T\in \T_h\}
\end{equation}
and
\begin{equation}\label{vh0space}
V^0_h=\{v: \ v\in V_h,\  v_b=0 \mbox{ on } \partial\Omega\}.
\end{equation}
We would like to emphasize that any function $v\in V_h$ has a single
value $v_b$ on each edge $e\in\E_h$.

For each element $T\in \T_h$, denote by $Q_0$ the $L^2$ projection
from $L^2(T)$ to $P_k(T)$  and by $Q_b$ the $L^2$ projection from
$L^2(e)$ to $P_{k-1}(e)$. Denote by $\bbQ_h$ the $L^2$ projection
onto the local discrete gradient space $[P_{k-1}(T)]^d$. Let
$V=H^1(\Omega)$. We define a projection operator $Q_h: V \to V_h$ so
that on each element $T\in\T_h$
\begin{equation}\label{Qh-def}
Q_h v=\{Q_0v_0, Q_bv_b\},\qquad \{v_0,v_b\}=i_W(v)\in W(T).
\end{equation}

Denote by $\nabla_{w,k-1}$ the discrete weak gradient operator on
the finite element space $V_h$ computed by using
(\ref{dwd}) on each element $T$; i.e.,
$$
(\nabla_{w,k-1}v)|_T =\nabla_{w,k-1, T} (v|_T),\qquad \forall v\in
V_h.
$$
For simplicity of notation, from now on we shall drop the subscript
$k-1$ in the notation $\nabla_{w,k-1}$ for the discrete weak
gradient.

Now we introduce two forms on $V_h$ as follows:
\begin{eqnarray*}
a(v,w) & = & \sum_{T\in {\cal T}_h}( a\nabla_w v, \nabla_w w)_T,\\
s(v,w) & = & \rho\sum_{T\in {\cal T}_h} h_T^{-1}\langle Q_bv_0-v_b,
Q_bw_0-w_b\rangle_{\partial T},
\end{eqnarray*}
where $\rho$ can be any positive number. In practical
computation, one might set $\rho=1$. Denote by $a_s(\cdot,\cdot)$ a
stabilization of $a(\cdot,\cdot)$ given by
$$
a_s(v,w)=a(v,w)+s(v,w).
$$

\begin{algorithm}
A numerical approximation for (\ref{pde}) and (\ref{bc}) can be
obtained by seeking $u_h=\{u_0,u_b\}\in V_h$ satisfying both $u_b=
Q_b g$ on $\partial \Omega$ and the following equation:
\begin{equation}\label{wg}
a_s(u_h,v)=(f,v_0), \quad\forall\ v=\{v_0,v_b\}\in V_h^0.
\end{equation}
\end{algorithm}

Note that the system (\ref{wg}) is symmetric and positive definite
for any parameter value of $\rho>0$.
\bigskip

Next, we justify the well-postedness of the scheme (\ref{wg}). For
any $v\in V_h$, let
\begin{equation}\label{3barnorm}
\3bar v\3bar:=\sqrt{a_s(v,v)}.
\end{equation}
It is not hard to see that $\3bar\cdot\3bar$ defines a semi-norm in
the finite element space $V_h$. We claim that this semi-norm becomes
to be a full norm in the finite element space $V_h^0$. It suffices
to check the positivity property for $\3bar\cdot\3bar$. To this end,
assume that $v\in V_h^0$ and $\3bar v\3bar=0$. It follows that
\[
(a\nabla_w v,\nabla_w v)+\rho\sum_{T\in\T_h} h_T^{-1}\langle
Q_bv_0-v_b, Q_bv_0-v_b\rangle_\pT=0,
\]
which implies that $\nabla_w v=0$ on each element $T$ and
$Q_bv_0=v_b$ on $\pT$. It follows from $\nabla_w v=0$ and
(\ref{dwd-2}) that for any $q\in [P_{k-1}(T)]^d$
\begin{eqnarray*}
0&=&(\nabla_w v,q)_T\\
&=&(\nabla v_0,q)_T-\langle v_0-v_b,q\cdot\bn\rangle_\pT\\
&=&(\nabla v_0,q)_T-\langle Q_bv_0-v_b,q\cdot\bn\rangle_\pT\\
&=&(\nabla v_0,q)_T.
\end{eqnarray*}
Letting $q=\nabla v_0$ in the equation above yields $\nabla v_0=0$
on $T\in {\cal T}_h$. Thus, $v_0=const$ on every $T\in\T_h$. This,
together with the fact that $Q_bv_0=v_b$ on $\partial T$ and $v_b=0$
on $\partial\Omega$, implies that $v_0=v_b=0$.

\medskip
\begin{lemma}
The weak Galerkin finite element scheme (\ref{wg}) has a unique
solution.
\end{lemma}

\smallskip

\begin{proof}
If $u_h^{(1)}$ and $u_h^{(2)}$ are two solutions of (\ref{wg}), then
$e_h=u_h^{(1)}-u_h^{(2)}$ would satisfy the following equation
$$
a_s(e_h,v)=0,\qquad\forall v\in V_h^0.
$$
Note that $e_h\in V_h^0$. Then by letting $v=e_h$ in the above
equation we arrive at
$$
\3bar e_h\3bar^2 = a_s(e_h, e_h)=0.
$$
It follows that $e_h\equiv 0$, or equivalently, $u_h^{(1)}\equiv
u_h^{(2)}$. This completes the proof of the lemma.
\end{proof}
\medskip

\section{Mass Conservation}\label{Section:wg-massconservation}

The second order elliptic equation (\ref{pde}) can be rewritten in a
conservative form as follows:
$$
\nabla \cdot q = f, \quad q=-a\nabla u.
$$
Let $T$ be any control volume. Integrating the first equation over
$T$ yields the following integral form of mass conservation:
\begin{equation}\label{conservation.01}
\int_{\partial T} q\cdot \bn ds = \int_T f dT.
\end{equation}
We claim that the numerical approximation from the weak Galerkin
finite element method (\ref{wg}) for (\ref{pde}) retains the mass
conservation property (\ref{conservation.01}) with an appropriately
defined numerical flux $q_h$. To this end, for any given $T\in {\cal
T}_h$, we chose in (\ref{wg}) a test function $v=\{v_0, v_b=0\}$ so
that $v_0=1$ on $T$ and $v_0=0$ elsewhere. It follows from
(\ref{wg}) that
\begin{equation}\label{mass-conserve.08}
\int_T a\nabla_{w} u_h\cdot \nabla_{w}v dT +\rho
h_T^{-1}\int_{\partial T} (Q_bu_0-u_b)ds = \int_T f dT.
\end{equation}
Let $\bbQ_h$ be the local $L^2$ projection onto the gradient space
$[P_{k-1}(T)]^d$. Using the definition (\ref{dwd}) for $\nabla_{w}v$
one arrives at
\begin{eqnarray}
\int_T a\nabla_{w} u_h\cdot \nabla_{w}v dT &=& \int_T
\bbQ_h(a\nabla_{w} u_h)\cdot \nabla_{w}v dT \nonumber\\
&=& - \int_T \nabla\cdot \bbQ_h(a\nabla_{w} u_h) dT \nonumber\\
&=& - \int_{\partial T} \bbQ_h(a\nabla_{w}u_h)\cdot\bn ds.
\label{conserv.88}
\end{eqnarray}
Substituting (\ref{conserv.88}) into (\ref{mass-conserve.08}) yields
\begin{equation}\label{mass-conserve.09}
\int_{\partial T} \left\{-\bbQ_h\left(a\nabla_{w}u_h\right)+\rho
h_T^{-1}(Q_bu_0-u_b)\bn\right\}\cdot\bn ds = \int_T f dT,
\end{equation}
which indicates that the weak Galerkin method conserves mass with a
numerical flux given by
$$
q_h = - \bbQ_h\left(a\nabla_{w}u_h\right)+\rho
h_T^{-1}(Q_bu_0-u_b)\bn.
$$

Next, we verify that the normal component of the numerical flux,
namely $q_h\cdot\bn$, is continuous across the edge of each element
$T$. To this end, let $e$ be an interior edge/face shared by two
elements $T_1$ and $T_2$. Choose a test function $v=\{v_0,v_b\}$ so
that $v_0\equiv 0$ and $v_b=0$ everywhere except on $e$. It follows
from (\ref{wg}) that
\begin{eqnarray}\label{mass-conserve.108}
\int_{T_1\cup T_2} a\nabla_{w} u_h\cdot \nabla_{w}v dT & & -\rho
h_{T_1}^{-1}\int_{\partial T_1\cap e} (Q_bu_0-u_b)|_{T_1}v_bds \\
& &- \rho h_{T_2}^{-1}\int_{\partial T_2\cap e}
(Q_bu_0-u_b)|_{T_2}v_bds\nonumber\\
& & =0.\nonumber
\end{eqnarray}
Using the definition of weak gradient (\ref{dwd}) we obtain
\begin{eqnarray*}
\int_{T_1\cup T_2} a\nabla_{w} u_h\cdot \nabla_{w}v dT&=&
\int_{T_1\cup T_2} \bbQ_h(a\nabla_{w} u_h)\cdot \nabla_{w}v dT\\
&=& \int_e\left(\bbQ_h(a\nabla_{w} u_h)|_{T_1}\cdot\bn_1 +
\bbQ_h(a\nabla_{w} u_h)|_{T_2}\cdot\bn_2\right)v_b ds,
\end{eqnarray*}
where $\bn_i$ is the outward normal direction of $T_i$ on the edge
$e$. Note that $\bn_1+\bn_2=0$. Substituting the above equation into
(\ref{mass-conserve.108}) yields
\begin{eqnarray*}
\int_e\left(-\bbQ_h(a\nabla_{w} u_h)|_{T_1}+\rho
h_{T_1}^{-1}(Q_bu_0-u_b)|_{T_1}\bn_1\right)\cdot\bn_1 v_bds\\
=-\int_e \left(-\bbQ_h(a\nabla_{w} u_h)|_{T_2}+\rho h_{T_2}^{-1}
(Q_bu_0-u_b)|_{T_2}\bn_2\right)\cdot\bn_2 v_bds,
\end{eqnarray*}
which shows the continuity of the numerical flux $q_h$ in the normal
direction.

\section{Some Technical Estimates}\label{Section:L2projections}

This section shall present some technical results useful for the
forthcoming error analysis. The first one is a trace inequality
established in \cite{wy-m} for functions on general shape regular
partitions. More precisely, let $T$ be an element with $e$ as an
edge. For any function $\varphi\in H^1(T)$, the following trace
inequality holds true (see \cite{wy-m} for details):
\begin{equation}\label{trace}
\|\varphi\|_{e}^2 \leq C \left( h_T^{-1} \|\varphi\|_T^2 + h_T
\|\nabla \varphi\|_{T}^2\right).
\end{equation}

Another useful result is a commutativity property for some
projection operators.
\begin{lemma}
Let $Q_h$ and $\bbQ_h$ be the $L^2$ projection operators defined in
previous sections. Then, on each element $T\in\T_h$, we have the
following commutative property
\begin{equation}\label{key}
\nabla_w (Q_h \phi) = \bbQ_h (\nabla \phi),\quad\forall \phi\in
H^1(T).
\end{equation}
\end{lemma}
\begin{proof}
Using (\ref{dwd}), the integration by
parts and the definitions of $Q_h$ and $\bbQ_h$, we have that for
any $\tau\in [P_{k-1}(T)]^d$
\begin{eqnarray*}
(\nabla_w (Q_h \phi),\tau)_T &=& -(Q_0 \phi,\nabla\cdot\tau)_T
+\langle Q_b \phi,\tau\cdot\bn\rangle_{\pT}\\
&=&-(\phi,\nabla\cdot\tau)_T + \langle \phi,\tau\cdot\bn\rangle_{\partial T}\\
&=&(\nabla \phi,\tau)_T\\
&=&(\bbQ_h(\nabla\phi),\tau)_T,
\end{eqnarray*}
which implies the desired identity (\ref{key}).
\end{proof}

\medskip
The following lemma provides some estimates for the projection
operators $Q_h$ and $\bbQ_h$. Observe that the underlying mesh
$\T_h$ is assumed to be sufficiently general to allow polygons or
polyhedra. A proof of the lemma can be found in \cite{wy-m}. It
should be pointed out that the proof of the lemma requires some
non-trivial technical tools in analysis, which have also been
established in \cite{wy-m}.

\begin{lemma}
Let $\T_h$ be a finite element partition of $\Omega$ that is shape
regular. Then, for any $\phi\in H^{k+1}(\Omega)$, we have
\begin{eqnarray}
&&\sum_{T\in\T_h} \|\phi-Q_0\phi\|_{T}^2 +\sum_{T\in\T_h}h_T^2
\|\nabla(\phi-Q_0\phi)\|_{T}^2\le C h^{2(k+1)}
\|\phi\|^2_{k+1},\label{Qh}\\
&&\sum_{T\in\T_h} \|a(\nabla\phi-\bbQ_h(\nabla\phi))\|^2_{T} \le
Ch^{2k} \|\phi\|^2_{k+1}.\label{Rh}
\end{eqnarray}
Here and in what follows of this paper, $C$ denotes a generic
constant independent of the meshsize $h$ and the functions in the
estimates.
\end{lemma}

\medskip
In the finite element space $V_h$, we introduce a discrete $H^1$
semi-norm as follows:
\begin{equation}\label{March24-2013-discreteH1norm}
\|v\|_{1,h} = \left( \sum_{T\in\T_h}\left(\|\nabla
v_0\|_T^2+h_T^{-1} \| Q_bv_0-v_b\|_\pT\right) \right)^{\frac12}.
\end{equation}
The following lemma indicates that $\|\cdot\|_{1,h}$ is equivalent
to the trip-bar norm (\ref{3barnorm}).

\begin{lemma} There exist two positive constants $C_1$ and $C_2$ such
that for any $v=\{v_0,v_b\}\in V_h$, we have
\begin{equation}\label{happy}
C_1 \|v\|_{1,h}\le \3bar v\3bar \leq C_2 \|v\|_{1,h}.
\end{equation}
\end{lemma}

\begin{proof}
For any $v=\{v_0,v_b\}\in V_h$, it follows from the definition of
weak gradient (\ref{dwd-2}) and $Q_b$ that
\begin{eqnarray}\label{March24-2013-01}
(\nabla_wv,q)_T=(\nabla v_0,q)_T+\l v_b-Q_bv_0,
q\cdot\bn\r_\pT,\quad \forall q\in [P_{k-1}(T)]^d.
\end{eqnarray}
By letting $q=\nabla_w v$ in (\ref{March24-2013-01}) we arrive at
\begin{eqnarray*}
(\nabla_wv,\nabla_w v)_T=(\nabla v_0,\nabla_w v)_T+\l v_b-Q_bv_0,
\nabla_w v\cdot\bn\r_\pT.
\end{eqnarray*}
From the trace inequality (\ref{trace}) and the inverse inequality
we have
\begin{eqnarray*}
(\nabla_wv,\nabla_w v)_T &\le& \|\nabla v_0\|_T \|\nabla_w v\|_T+ \|
Q_bv_0-v_b\|_\pT \|\nabla_w v\|_\pT\\
&\le& \|\nabla v_0\|_T \|\nabla_w v\|_T+ Ch_T^{-1/2}\|
Q_bv_0-v_b\|_\pT \|\nabla_w v\|_T\\
\end{eqnarray*}
Thus,
$$
\|\nabla_w v\|_T \le C \left(\|\nabla v_0\|_T^2 +h_T^{-1}\|
Q_bv_0-v_b\|_\pT^2\right)^{\frac12},
$$
which verifies the upper bound of $\3bar v\3bar$. As to the lower
bound, we chose $q=\nabla v_0$ in (\ref{March24-2013-01}) to obtain
\begin{eqnarray*}
(\nabla_wv,\nabla v_0)_T=(\nabla v_0,\nabla v_0)_T+\l v_b-Q_bv_0,
\nabla v_0\cdot\bn\r_\pT.
\end{eqnarray*}
Thus, from the trace an inverse inequality we have
$$
\|\nabla v_0\|_T^2 \leq \|\nabla_w v\|_T \|\nabla v_0\|_T
+Ch_T^{-1/2}\| Q_bv_0-v_b\|_\pT \|\nabla v_0\|_T.
$$
This leads to
$$
\|\nabla v_0\|_T \leq C\left(\|\nabla_w v\|_T^2 +Ch_T^{-1}\|
Q_bv_0-v_b\|_\pT^2\right)^{\frac12},
$$
which verifies the lower bound for $\3bar v\3bar$. Collectively,
they complete the proof of the lemma.
\end{proof}

\bigskip

\begin{lemma} Assume that $\T_h$ is shape regular. Then for any $w\in H^{k+1}(\Omega)$ and
$v=\{v_0,v_b\}\in V_h$, we have
\begin{eqnarray}
|s(Q_hw, v)|&\le&
Ch^k\|w\|_{k+1}\3bar v\3bar,\label{mmm1}\\
\left|\ell_w(v)\right| &\leq& C h^k\|w\|_{k+1} \3bar
v\3bar,\label{mmm2}
\end{eqnarray}
where  $\ell_w(v)=\sum_{T\in\T_h} \langle a(\nabla w-\bbQ_h\nabla w)\cdot\bn,\;
v_0-v_b\rangle_\pT$.
\end{lemma}

\medskip

\begin{proof}
Using the definition of $Q_b$, (\ref{trace}), and (\ref{Qh}), we
obtain
\begin{eqnarray*}
|s(Q_hw, v)|&=&
\left|\sum_{T\in\T_h} h_T^{-1}\langle Q_b(Q_0w)-Q_bw,\; Q_bv_0-v_b\rangle_\pT\right|\\
 &=&\left|\sum_{T\in\T_h} h_T^{-1}\langle Q_0w-w,\; Q_bv_0-v_b\rangle_\pT\right|\\
&\le& C\left(\sum_{T\in\T_h}(h_T^{-2}\|Q_0w-w\|_T^2+
\|\nabla (Q_0w-w)\|_T^2)\right)^{\frac12}\cdot\\ &&
\left(\sum_{T\in\T_h}h_T^{-1}\|Q_bv_0-v_b\|^2_{\pT}\right)^{\frac12}\\
&\le& Ch^k\|w\|_{k+1}\3bar v\3bar.
\end{eqnarray*}

As to (\ref{mmm2}), it follows from the Cauchy-Schwarz inequality,
the trace inequality (\ref{trace}) and the estimate (\ref{Rh}) that
\begin{eqnarray}\label{happy1}
|\ell_w(v)|&=&\left|\sum_{T\in\T_h}\langle a(\nabla w-\bbQ_h\nabla
w)\cdot\bn, v_0-v_b\rangle_\pT\right|\\
&\le & C \sum_{T\in\T_h}\|a(\nabla w-\bbQ_h\nabla w)\|_{\pT}
\|v_0-v_b\|_\pT\nonumber\\
&\le & C \left(\sum_{T\in\T_h}h_T\|a(\nabla w-\bbQ_h\nabla
w)\|_{\pT}^2\right)^{\frac12}
\left(\sum_{T\in\T_h}h_T^{-1}\|v_0-v_b\|_\pT^2\right)^{\frac12}\nonumber\\
&\le &
Ch^k\|w\|_{k+1}\left(\sum_{T\in\T_h}h_T^{-1}\|v_0-v_b\|_\pT^2\right)^{\frac12}.\nonumber
\end{eqnarray}
Using the trace inequality (\ref{trace}) and the approximation
property of the $L^2$ projection operator we obtain
\begin{eqnarray*}
\|v_0-v_b\|_\pT &\leq& \|v_0-Q_b v_0\|_\pT + \|Q_b v_0-v_b\|_\pT\\
&\le& C h_T^{1/2}\|\nabla v_0\|_T +\|Q_b v_0-v_b\|_\pT.
\end{eqnarray*}
Substituting the above inequality into (\ref{happy1}) yields
\begin{eqnarray}\label{happy2}
|\ell_w(v)| \leq Ch^k\|w\|_{k+1}\left(\sum_{T\in\T_h}\left\{\|\nabla
v_0\|_T^2 + h_T^{-1}\|Q_bv_0-v_b\|_\pT^2\right\}\right)^{\frac12},
\end{eqnarray}
which, along with the estimate (\ref{happy}), verifies the desired
estimate (\ref{mmm2}).
\end{proof}

\section{Error Analysis}\label{Section:error-analysis}
The goal of this section is to establish some error estimates for
the weak Galerkin finite element solution $u_h$ arising from (\ref{wg}).
The error will be measured in two natural norms: the triple-bar
norm as defined in (\ref{3barnorm}) and the standard $L^2$ norm. The
triple bar norm is essentially a discrete $H^1$ norm for the
underlying weak function.

For simplicity of analysis, we assume that the coefficient tensor
$a$ in (\ref{pde}) is a piecewise constant matrix with respect to
the finite element partition $\T_h$. The result can be extended to
variable tensors without any difficulty, provided that the tensor
$a$ is piecewise sufficiently smooth.

\subsection{Error equation}
Let $u_h=\{u_0,u_b\}\in V_h$ be the weak
Galerkin finite element solution arising from the numerical scheme
(\ref{wg}). Assume that the exact solution of
(\ref{pde})-(\ref{bc}) is given by $u$. The $L^2$
projection of $u$ in the finite element space $V_h$ is given by
$$
Q_h u=\{Q_0 u,Q_b u\}.
$$
Let
$$
e_h=\{e_0,e_b\}=\{Q_0u-u_0,Q_bu-u_b\}
$$
be the error between the WG finite element solution and the $L^2$
projection of the exact solution.

\begin{lemma}\label{Lemma:error-equation}
Let $e_h$  be the error of the weak Galerkin
finite element solution arising from (\ref{wg}). Then,
for any $v\in V_h^0$ we have
\begin{eqnarray}
a_s(e_h,v)=\ell_u(v)+ s(Q_hu,v),\label{ee}
\end{eqnarray}
where $\ell_u(v)=\sum_{T\in\T_h} \langle a(\nabla u-\bbQ_h\nabla
u)\cdot\bn,v_0-v_b\rangle_\pT$.
\end{lemma}

\begin{proof}
Testing (\ref{pde}) by using $v_0$ of $v=\{v_0,v_b\}\in V_h^0$ we
arrive at
\begin{equation}\label{m1}
\sum_{T\in\T_h}(a\nabla u,\nabla v_0)_T-\sum_{T\in\T_h} \langle
a\nabla u\cdot\bn,v_0-v_b\rangle_\pT=(f,v_0),
\end{equation}
where we have used the fact that $\sum_{T\in\T_h}\langle a\nabla
u\cdot\bn, v_b\rangle_\pT=0$. To deal with the term
$\sum_{T\in\T_h}(a\nabla u,\nabla v_0)_T$ in (\ref{m1}), we need the
following equation. For any $\phi\in H^1(T)$ and $v\in V_h$, it
follows from (\ref{key}), the definition of the discrete weak
gradient (\ref{dwd}), and the integration by parts that
\begin{eqnarray}
(a\nabla_w Q_h\phi,\nabla_w v)_T&=&(a \bbQ_h(\nabla\phi),\nabla_w v)_T\nonumber\\
&=& -(v_0,\nabla\cdot (a \bbQ_h\nabla\phi))_T+\langle v_b,(a \bbQ_h\nabla\phi)\cdot\bn\rangle_\pT\nonumber\\
&=&(\nabla v_0,a\bbQ_h\nabla\phi)_T-\langle v_0-v_b,(a\bbQ_h\nabla\phi)\cdot\bn\rangle_\pT\nonumber\\
&=&(a\nabla\phi,\nabla v_0)_T-\l (a\bbQ_h\nabla\phi)\cdot\bn,\
v_0-v_b\r_\pT.\label{j1}
\end{eqnarray}
By letting $\phi=u$ in (\ref{j1}),
we have from combining (\ref{j1}) and (\ref{m1}) that
\begin{eqnarray*}
\sum_{T\in\T_h} (a\nabla_w Q_hu,\nabla_w v)_T&=&(f,v_0)+
\sum_{T\in\T_h} \langle a(\nabla u-\bbQ_h\nabla
u)\cdot\bn,v_0-v_b\rangle_\pT\\
&=&(f,v_0)+\ell_u(v).
\end{eqnarray*}
Adding $s(Q_hu,v)$ to both sides of the above equation gives
\begin{equation}\label{j2}
a_s(Q_hu, v)=(f, v_0)+ \ell_u(v) +s(Q_hu,v).
\end{equation}
Subtracting (\ref{wg}) from (\ref{j2}) yields the following error
equation,
\begin{eqnarray*}
a_s(e_h,v)=\ell_u(v)+ s(Q_hu,v),\quad \forall v\in V_h^0.
\end{eqnarray*}
This completes the proof of the lemma.
\end{proof}

\subsection{Error estimates}
The error equation (\ref{ee}) can be used to derive the following
error estimate for the WG finite element solution.

\begin{theorem} Let $u_h\in V_h$ be the weak Galerkin finite element solution of the problem
(\ref{pde})-(\ref{bc}) arising from (\ref{wg}). Assume the exact solution $u\in H^{k+1}(\Omega)$. Then,
there exists a constant $C$ such that
\begin{equation}\label{err1}
\3bar u_h-Q_hu\3bar \le Ch^{k}\|u\|_{k+1}.
\end{equation}
\end{theorem}
\begin{proof}
By letting $v=e_h$ in (\ref{ee}), we have
\begin{eqnarray}
\3bar e_h\3bar^2&=&\ell_u(e_h)+s(Q_hu,\;\ e_h).\label{main}
\end{eqnarray}
It then follows from (\ref{mmm1}) and (\ref{mmm2}) that
\[
\3bar e_h\3bar^2 \le Ch^k\|u\|_{k+1}\3bar e_h\3bar,
\]
which implies (\ref{err1}). This completes the proof.
\end{proof}

Next, we will measure the difference between $u$ and $u_h$ in the
discrete $H^1$ semi-norm $\|\cdot\|_{1,h}$ as defined in
(\ref{March24-2013-discreteH1norm}). Note that
(\ref{March24-2013-discreteH1norm}) can be easily extended to
functions in $H^1(\Omega)+V_h$ through the inclusion map $i_W$.

\begin{corollary}
 Let $u_h\in V_h$ be the weak Galerkin finite element solution of the problem
(\ref{pde})-(\ref{bc}) arising from (\ref{wg}). Assume the exact solution $u\in H^{k+1}(\Omega)$. Then,
there exists a constant $C$ such that
\begin{equation}\label{err8}
\| u-u_h\|_{1,h} \le Ch^{k}\|u\|_{k+1}.
\end{equation}
\end{corollary}

\begin{proof}
It follows from (\ref{happy}) and (\ref{err1}) that
$$
\|Q_h u-u_h\|_{1,h}\le C\3bar Q_hu-u_h \3bar\le Ch^{k}\|u\|_{k+1}.
$$
Using the triangle inequality, (\ref{Qh}) and  the equation above, we have
$$
\| u-u_h\|_{1,h}\le\| u-Q_hu\|_{1,h}+\| Q_hu-u_h\|_{1,h}\le
Ch^{k}\|u\|_{k+1}.
$$
This completes the proof.
\end{proof}

\bigskip

In the rest of the section, we shall derive an optimal order error
estimate for the weak Galerkin finite element scheme (\ref{wg}) in
the usual $L^2$ norm by using a duality argument as was commonly
employed in the standard Galerkin finite element methods \cite{ci,
sue}. To this end, we consider a dual problem that seeks $\Phi\in
H_0^1(\Omega)$ satisfying
\begin{eqnarray}
-\nabla\cdot (a \nabla\Phi)&=& e_0\quad
\mbox{in}\;\Omega.\label{dual}
\end{eqnarray}
Assume that the above dual problem has the usual $H^{2}$-regularity.
This means that there exists a constant $C$ such that
\begin{equation}\label{reg}
\|\Phi\|_2\le C\|e_0\|.
\end{equation}

\begin{theorem} Let $u_h\in V_h$ be the weak Galerkin finite element solution of the problem
(\ref{pde})-(\ref{bc}) arising from (\ref{wg}). Assume the exact solution $u\in H^{k+1}(\Omega)$. In
addition, assume that the dual problem (\ref{dual}) has the usual
$H^2$-regularity. Then, there exists a constant $C$ such that
\begin{equation}\label{err2}
\|u-u_0\| \le Ch^{k+1}\|u\|_{k+1}.
\end{equation}
\end{theorem}

\begin{proof}
By testing (\ref{dual}) with $e_0$ we obtain
\begin{eqnarray}\nonumber
\|e_0\|^2&=&-(\nabla\cdot (a\nabla\Phi),e_0)\\
&=&\sum_{T\in\T_h}(a\nabla \Phi,\ \nabla e_0)_T-\sum_{T\in\T_h}\l
a\nabla\Phi\cdot\bn,\ e_0- e_b\r_{\pT},\label{jw.08}
\end{eqnarray}
where we have used the fact that $e_b=0$ on $\partial\Omega$.
Setting $\phi=\Phi$ and $v=e_h$ in (\ref{j1}) yields
\begin{eqnarray}
(a\nabla_w Q_h\Phi,\;\nabla_w e_h)_T=(a\nabla\Phi,\;\nabla e_0)_T-\l
(a\bbQ_h\nabla\Phi)\cdot\bn,\ e_0-e_b\r_\pT.\label{j1-new}
\end{eqnarray}
Substituting (\ref{j1-new}) into (\ref{jw.08}) gives
\begin{eqnarray}
\|e_0\|^2&=&(a\nabla_w e_h,\ \nabla_w Q_h\Phi)+\sum_{T\in\T_h}\l
a(\bbQ_h\nabla\Phi-\nabla\Phi)\cdot\bn,\ e_0-e_b\r_{\pT}\nonumber\\
&=&(a\nabla_w e_h,\ \nabla_w Q_h\Phi)+\ell_\Phi(e_h).\label{m2}
\end{eqnarray}
It follows from the error equation (\ref{ee}) that
\begin{eqnarray}
(a\nabla_w e_h,\ \nabla_w Q_h\Phi)&=&\ell_u(Q_h\Phi)
+s(Q_hu,\ Q_h\Phi)-s(e_h,\ Q_h\Phi).\label{m3}
\end{eqnarray}
By combining (\ref{m2}) with (\ref{m3}) we arrive at
\begin{eqnarray}\label{m4}
\|e_0\|^2=\ell_u(Q_h\Phi)
+s(Q_hu,\ Q_h\Phi)-s(e_h,\ Q_h\Phi)
+\ell_\Phi(e_h).
\end{eqnarray}

Let us bound the terms on the right hand side of (\ref{m4}) one by
one. Using the triangle inequality, we obtain
\begin{eqnarray}
|\ell_u(Q_h\Phi)|&=&\left|\sum_{T\in\T_h} \langle a(\nabla u-\bbQ_h\nabla
u)\cdot\bn,\; Q_0\Phi-Q_b\Phi\rangle_\pT \right|\nonumber\\
&\le&\left|\sum_{T\in\T_h} \langle a(\nabla u-\bbQ_h\nabla
u)\cdot\bn,\; Q_0\Phi-\Phi\rangle_\pT \right|\nonumber\\
&+&\left|\sum_{T\in\T_h} \langle a(\nabla u-\bbQ_h\nabla
u)\cdot\bn,\; \Phi-Q_b\Phi\rangle_\pT \right|.\label{1st-term}
\end{eqnarray}
We first use the definition of $Q_b$ and
the fact that $\Phi=0$ on $\partial\Omega$ to obtain
\begin{eqnarray}
\sum_{T\in\T_h}\l a(\nabla u-\bbQ_h\nabla u)\cdot\bn, \Phi-Q_b\Phi\r_\pT =\sum_{T\in\T_h}\l a\nabla u\cdot\bn, \Phi-Q_b\Phi\r_\pT = 0.\label{l21}
\end{eqnarray}
From the trace inequality (\ref{trace}) and the estimate (\ref{Qh})
we have
$$
\left(\sum_{T\in\T_h}\|Q_0\Phi-\Phi\|^2_\pT\right)^{1/2} \leq C
h^{\frac32}\|\Phi\|_2
$$
and
$$
\left(\sum_{T\in\T_h}\|a(\nabla u-\bbQ_h\nabla
u)\|^2_\pT\right)^{1/2} \leq Ch^{k-\frac12}\|u\|_{k+1}.
$$
Thus, it follows from the Cauchy-Schwarz inequality and the above
two estimates that
\begin{eqnarray}
&&\left|\sum_{T\in\T_h} \langle a(\nabla u-\bbQ_h\nabla
u)\cdot\bn,\; Q_0\Phi-\Phi\rangle_\pT \right|\nonumber\\
&&\le C\left(\sum_{T\in\T_h}\|a(\nabla u-\bbQ_h\nabla
u)\|^2_\pT\right)^{1/2}
\left(\sum_{T\in\T_h}\|Q_0\Phi-\Phi\|^2_\pT\right)^{1/2} \nonumber\\
&&\le  C h^{k+1} \|u\|_{k+1}\|\Phi\|_2.\label{l22}
\end{eqnarray}
Combining (\ref{1st-term}) with (\ref{l21}) and (\ref{l22}) yields
\begin{eqnarray}\label{1st-term-complete}
|\ell_u(Q_h\Phi)| \leq C h^{k+1} \|u\|_{k+1}
\|\Phi\|_2.
\end{eqnarray}
Analogously, it follows from the definition of $Q_b$, the trace
inequality (\ref{trace}), and the estimate (\ref{Qh}) that
\begin{eqnarray}\nonumber
\left|s(Q_hu,\; Q_h\Phi)\right|&\le & \rho\sum_{T\in\T_h}h_T^{-1}
\left|(Q_b(Q_0u)-Q_bu,\ Q_b(Q_0\Phi)-Q_b\Phi)_\pT\right|\nonumber\\
&\le & \rho\sum_{T\in\T_h}h_T^{-1}\|Q_b(Q_0u-u)\|_\pT\|Q_b(Q_0\Phi-\Phi)\|_\pT\nonumber\\
&\le & \rho\sum_{T\in\T_h}h_T^{-1}\|Q_0u-u\|_\pT\|Q_0\Phi-\Phi\|_\pT\nonumber\\
&\le& C\left(\sum_{T\in\T_h}h_T^{-1}\|Q_0u-u\|^2_\pT\right)^{1/2}
\left(\sum_{T\in\T_h}h_T^{-1}\|Q_0\Phi-\Phi\|^2_\pT\right)^{1/2}\nonumber  \\
&\le& Ch^{k+1}\|u\|_{k+1}\|\Phi\|_2.\label{2nd-term-complete}
\end{eqnarray}
The estimates (\ref{mmm1}) with $k=1$ and the error estimate
(\ref{err1}) imply
\begin{eqnarray}\label{3rd-term-complete}
|s(e_h,\ Q_h\Phi)|\le Ch\|\Phi\|_2\3bar e_h\3bar\le
Ch^{k+1}\|u\|_{k+1}\|\Phi\|_2.
\end{eqnarray}
Similarly, it follows from (\ref{mmm2}) and (\ref{err1}) that
\begin{eqnarray}\label{4th-term-complete}
|\ell_\Phi(e_h)| &\le&  Ch^{k+1}\|u\|_{k+1}\|\Phi\|_2.
\end{eqnarray}
Now substituting (\ref{1st-term-complete})-(\ref{4th-term-complete})
into (\ref{m4}) yields
$$
\|e_0\|^2 \leq C h^{k+1}\|u\|_{k+1} \|\Phi\|_2,
$$
which, combined with the regularity assumption (\ref{reg}) and the
triangle inequality, gives the desired optimal order error estimate
(\ref{err2}).
\end{proof}

\section{Numerical Examples}\label{Section:numerical}
In this section, we examine the WG method by testing its convergence
and flexibility for solving second order elliptic problems. In the
test of convergence, the first ($k=1$) and second ($k=2$) order of
weak Galerkin elements are used in the construction of the finite
element space $V_h$. In the test of flexibility of the WG method,
elliptic problems are solved on finite element partitions with
various configurations, including triangular mesh, deformed
rectangular mesh, and honeycomb mesh in two dimensions and deformed
cubic mesh in three dimensions. Our numerical results confirm the
theory developed in previous sections; namely, optimal rate of
convergence in $H^1$ and $L^2$ norms. In addition, it shows a great
flexibility of the WG method with respect to the shape of finite
element partitions.

Let $u_h=\{u_0,u_b\}$ and $u$ be the solution to the weak Galerkin
equation and the original equation, respectively. The error is
defined by $e_h=u_h-Q_hu=\{e_0,e_b\}$, where $e_0=u_0-Q_0u$ and
$e_b=u_b-Q_bu$. Here $Q_h u=\{Q_0 u,Q_b u\}$ with $Q_h$ as the $L^2$
projection onto appropriately defined spaces. The following norms
are used to measure the error in all of the numerical experiments:
\begin{eqnarray*}
&&H^1\mbox{ semi-norm: } \3bar e_h\3bar=
\bigg(\sum_{T\in\mathcal{T}_h}\int_T|\nabla_w e_h|^2dT
+h^{-1}\int_{\partial T}|Q_be_0-e_b|^2ds\bigg)^{\frac12},\\
&&\mbox{ Element-based }L^2\mbox{ norm: }\|e_0\|
=\bigg(\sum_{T\in\mathcal{T}_h}\int_K|e_0|^2dT\bigg)^{\frac12}.
\end{eqnarray*}

\subsection{On Triangular Mesh}
Consider the second order elliptic equation that seeks an unknown function $u=u(x,y)$ satisfying
$$-\nabla\cdot(a\nabla)=f$$
in the square domain $\Omega=(0,1)\times(0,1)$ with Dirichlet boundary condition. The boundary condition $u|_{\partial\Omega}=g$ and $f$ are chosen such that the exact solution is given by $u=\sin(\pi x)\cos(\pi y)$ and
$$a=\begin{pmatrix}
x^2+y^2+1&xy\\
xy&x^2+y^2+1
\end{pmatrix}.$$

The triangular mesh  $\T_h$ used in this example is constructed by:
1) uniformly partitioning the domain into $n\times n$
sub-rectangles; 2) dividing each rectangular element by the diagonal
line with a negative slope. The mesh size is denoted by $h=1/n$. The
lowest order ($k=1$) weak Galerkin element is used for obtaining the
weak Galerkin solution $u_h=\{u_0,u_b\}$; i.e., $u_0$ and $u_b$ are
polynomials of degree $k=1$ and degree $k-1=0$ respectively on each
element $T\in \T_h$.

Table \ref{tab:ex1} shows the convergence rate for WG solutions
measured in $H^1$ and $L^2$ norms.  The numerical results indicate
that the WG solution of linear element is convergent with rate
$O(h)$ in $H^1$ and $O(h^2)$ in $L^2$ norms.

\begin{table}
  \caption{Example 1. Convergence rate of lowest order WG ($k=1$) on triangular meshes.}
  \label{tab:ex1}
  \center
  \begin{tabular}{||c|c|c|c|c||}
    \hline\hline
   $h$ & $\3bar e_h \3bar$ & order & $\| e_0\|$ & order \\
    \hline
   1/4   &  1.3240e+00 &            & 1.5784e+00 &        \\ \hline
   1/8   &  6.6333e-01 & 9.9710e-01 & 3.6890e-01 & 2.0972 \\ \hline
   1/16  &  3.3182e-01 & 9.9933e-01 & 9.0622e-02 & 2.0253 \\ \hline
   1/32  &  1.6593e-01 & 9.9983e-01 & 2.2556e-02 & 2.0064 \\ \hline
   1/64  &  8.2966e-02 & 9.9998e-01 & 5.6326e-03 & 2.0016 \\ \hline
   1/128 &  4.1483e-02 & 1.0000     & 1.4078e-03 & 2.0004\\ \hline\hline
   \end{tabular}
\end{table}

In the second example, we consider the Poisson problem that seeks an
unknown function $u=u(x,y)$ satisfying
$$
-\Delta u=f
$$
in the square domain $\Omega=(0,1)^2$. Like the first example, the
exact solution here is given by $u=\sin(\pi x)\cos(\pi y)$ and $g$
and $f$ are chosen accordingly to match the exact solution.

The very same triangular mesh is employed in the numerical
calculation. Associated with this triangular mesh $\T_h$, two weak
Galerkin elements with $k=1$ and $k=2$ are used in the computation
of the weak Galerkin finite element solution $u_h$. For simplicity,
these two elements shall be referred to as $(P_1(T), P_{0}(e))$ and
$(P_2(T), P_{1}(e))$.

Tables \ref{tab:ex1_1} and \ref{tab:ex1_2} show the numerical
results on rate of convergence for the WG solutions in $H^1$ and
$L^2$ norms associated with $k=1$ and $k=2$, respectively. Note that
$\|e_h\|_{\mathcal{E}_h}$ is a discrete $L^2$ norm for the
approximation $u_b$ on the boundary of each element. Optimal rates
of convergence are observed numerically for each case.

\begin{table}
  \caption{Example 2. Convergence rate of lowest order WG ($k=1$) on triangular meshes.}
  \label{tab:ex1_1}
  \center
  \begin{tabular}{||c||c|c|c||}
    \hline\hline
   $h$ & $\3bar e_h\3bar$ & $\|e_h\|$ & $\|e_h\|_{\mathcal{E}_h}$  \\
    \hline\hline
   1/2     &2.7935e-01   &6.1268e-01   &5.7099e-02    \\ \hline
   1/4     &1.4354e-01   &1.5876e-01   &1.3892e-02    \\ \hline
   1/8     &7.2436e-02   &4.0043e-02   &3.5430e-03    \\ \hline
   1/16    &3.6315e-02   &1.0033e-02   &8.9325e-04    \\ \hline
   1/32    &1.8170e-02   &2.5095e-03   &2.2384e-04    \\ \hline
   1/64    &9.0865e-03   &6.2747e-04   &5.5994e-05    \\ \hline
   1/128   &4.5435e-03   &1.5687e-04   &1.4001e-05    \\ \hline\hline
   $O(h^r),r=$  &9.9232e-01   &1.9913       &1.9955   \\ \hline\hline
   \end{tabular}
\end{table}

\begin{table}
  \caption{Example 2. Convergence rate of second order WG ($k=2$) on triangular meshes.}
  \label{tab:ex1_2}
  \center
  \begin{tabular}{||c||c|c|c||}
    \hline\hline
   $h$ & $\3bar e_h\3bar$ & $\|e_h\|$ & $\|e_h\|_{\mathcal{E}_h}$    \\
    \hline\hline

   1/2     &1.7886e-01   &9.4815e-02   &3.3742e-02    \\ \hline
   1/4     &4.8010e-02   &1.2186e-02   &4.9969e-03    \\ \hline
   1/8     &1.2327e-02   &1.5271e-03   &6.6539e-04    \\ \hline
   1/16    &3.1139e-03   &1.9077e-04   &8.5226e-05    \\ \hline
   1/32    &7.8188e-04   &2.3829e-05   &1.0763e-05    \\ \hline
   1/64    &1.9586e-04   &2.9774e-06   &1.3516e-06    \\ \hline
   1/128   &4.9009e-05   &3.7210e-07   &1.6932e-07    \\ \hline\hline
   $O(h^r),r=$  &1.9769   &2.9956       &2.9453       \\ \hline\hline
      \end{tabular}
\end{table}

\subsection{On Quadrilateral Meshes}

In this test, we solve the same poisson equation considered in the
second example by using quadrilateral meshes. We start with an
initial quadrilateral mesh, shown as in Figure \ref{fig:ex2} (Left).
The mesh is then successively refined by connecting the barycenter
of each coarse element with the middle points of its edges, shown as
in Figure \ref{fig:ex2} (Right). For the quadrilateral mesh $\T_h$,
two weak Galerkin elements with $k=1$ and $k=2$ are used in the WG
finite element scheme (\ref{wg}).

Tables \ref{tab:ex2} and \ref{ex2_2} show the rate of convergence
for the WG solutions in $H^1$ and $L^2$ norms associated with $k=1$
and $k=2$ on quadrilateral meshes, respectively. Optimal rates of
convergence are observed numerically.

\begin{figure}[!htb]
\centering
\begin{tabular}{cc}
  \resizebox{2in}{1.8in}{\includegraphics{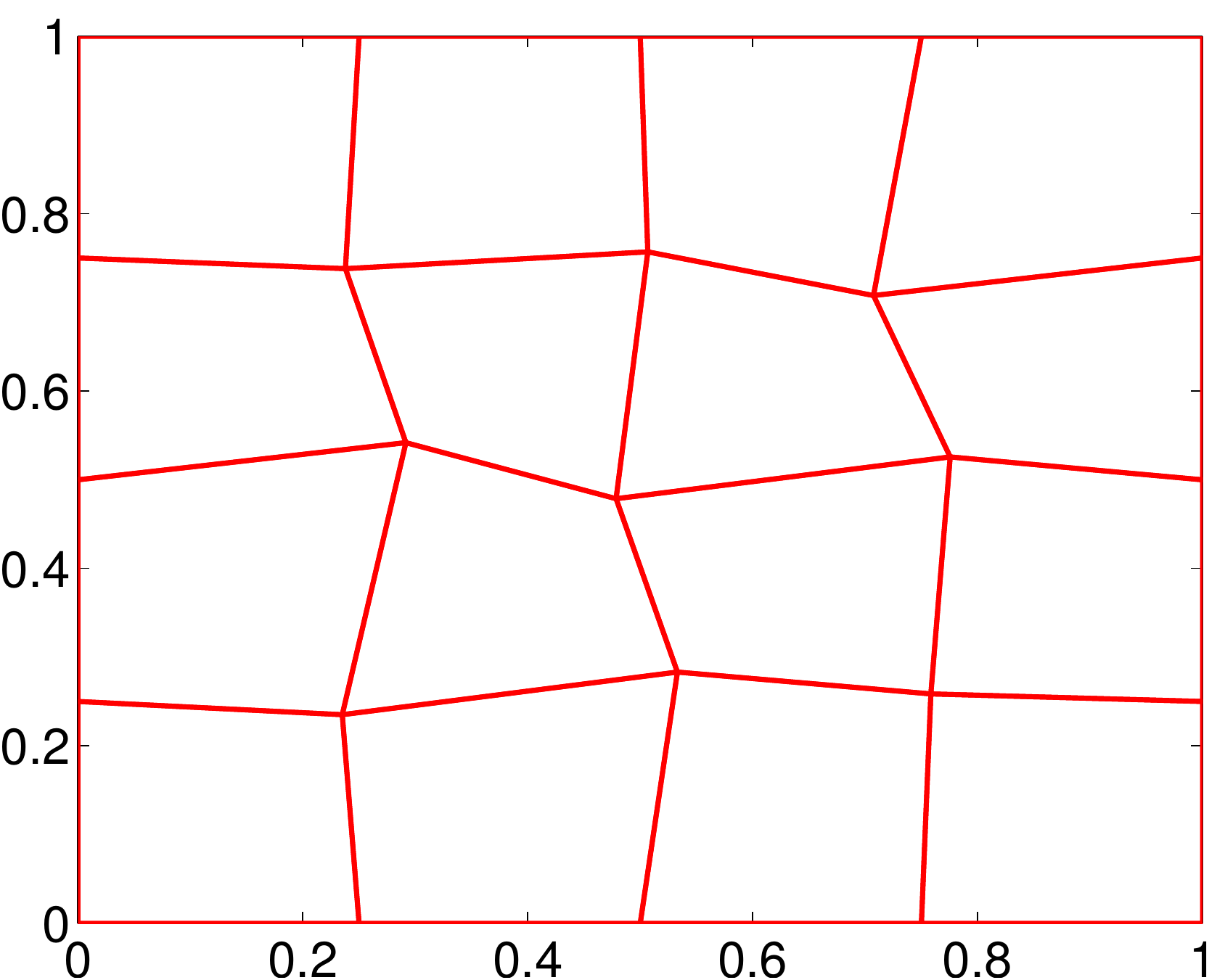}} \quad
  \resizebox{2in}{1.8in}{\includegraphics{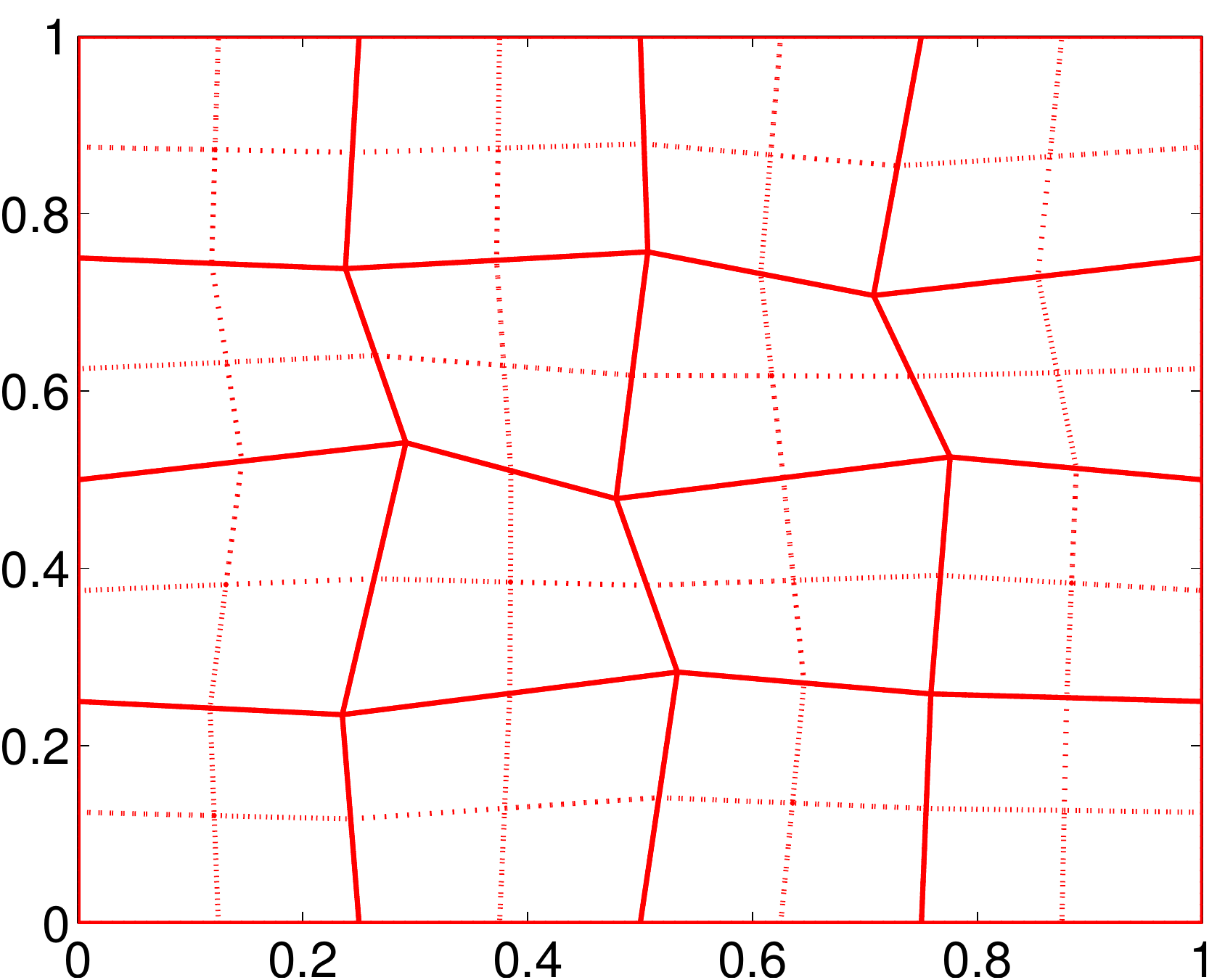}}
\end{tabular}
\caption{Mesh level 1 (Left) and mesh level 2 (Right) for example
2.}\label{fig:ex2}
\end{figure}

\begin{table}[!htb]
  \caption{Example 3. Error and rate of convergence for first order WG on quadrilateral meshes.}
  \label{tab:ex2}
  \center
  \begin{tabular}{||c|c|c|c|c||}
      \hline\hline
   $h$ & $\3bar e_h \3bar$ & order & $\| e_0\|$ & order \\
    \hline
   2.9350e-01 &  1.9612e+00 &            &2.1072e+00 &        \\ \hline
   1.4675e-01 &  1.0349e+00 &9.2225e-01  &5.7219e-01 & 1.8808 \\ \hline
   7.3376e-02 &  5.2434e-01 &9.8094e-01  &1.4458e-01 & 1.9847 \\ \hline
   3.6688e-02 &  2.6323e-01 &9.9418e-01  &3.5655e-02 & 2.0197 \\ \hline
   1.8344e-02 &  1.3179e-01 &9.9808e-01  &8.6047e-03 & 2.0509 \\ \hline
   9.1720e-03 &  6.5925e-02 &9.9934e-01  &2.0184e-03 & 2.0919 \\ \hline\hline
   \end{tabular}
\end{table}

\begin{table}
  \caption{Example 3. Error and rate of convergence for second order WG on quadrilateral meshes.}
  \label{ex2_2}
  \center
  \begin{tabular}{||c||c|c|c|c|c|c||}
    \hline\hline
   $h$ & $\3bar e_h\3bar$ & order & $\|e_0\|$  & order \\
    \hline\hline
   1/2     &1.7955e-01 &        &1.4891e-01  &           \\ \hline
   1/4     &8.7059e-02 &1.0444  &1.8597e-02  &3.0013     \\ \hline
   1/8     &2.8202e-02 &1.6262  &2.1311e-03  &3.1254     \\ \hline
   1/16    &7.8114e-03 &1.8521  &2.4865e-04  &3.0995     \\ \hline
   1/32    &2.0347e-03 &1.9408  &2.9964e-05  &3.0528     \\ \hline
   1/64    &5.1767e-04 &1.9747  &3.6806e-06  &3.0252     \\ \hline
   1/128   &1.3045e-04 &1.9885  &4.5627e-07  &3.0120     \\ \hline\hline
      \end{tabular}
\end{table}

\subsection{On Honeycomb Mesh}
In the forth test, we solve the Poisson equation on the domain of
unit square with exact solution $u=\sin(\pi x)\sin(\pi y)$. The
Dirichlet boundary data $g$ and $f$ are chosen to match the exact
solution. The numerical experiment is performed on the honeycomb
mesh as shown in Figure \ref{fig:ex3}. The linear WG element ($k=1$)
is used in this numerical computation.

The error profile is presented in Table \ref{tab:ex3}, which
confirms the convergence rates predicted by the theory.

\begin{figure}[!htb]
\centering
\begin{tabular}{c}
  \resizebox{2.2in}{1.9in}{\includegraphics{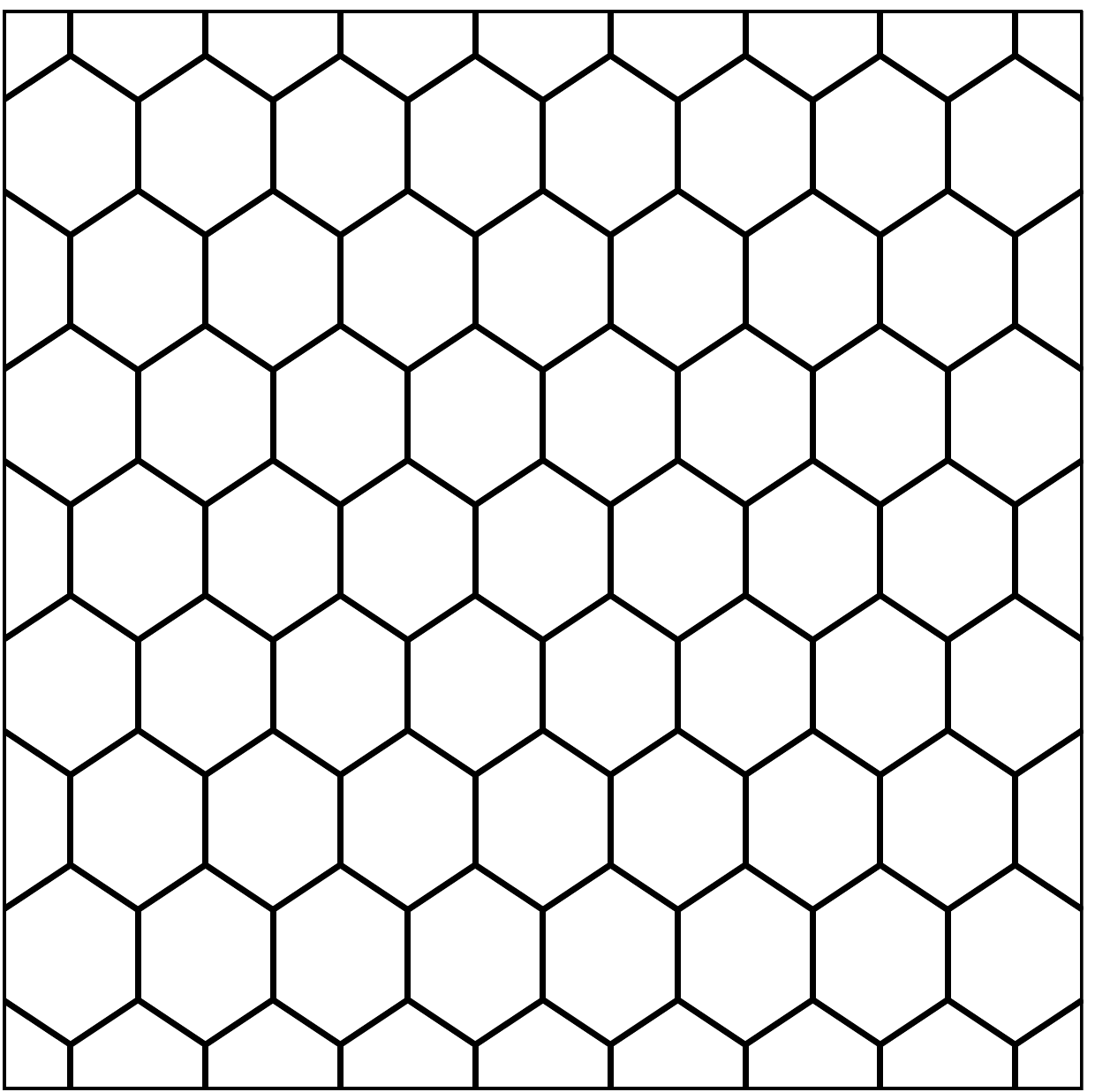}}
\end{tabular}
\caption{Honeycomb mesh for example 3.}\label{fig:ex3}
\end{figure}

\begin{table}[!htb]
  \caption{Example 4. Error and rate of convergence for linear WG element on honeycomb meshes.}\label{tab:ex3}
  \center
  \begin{tabular}{||c|c|c|c|c||}
      \hline\hline
   $h$ & $\3bar e_h \3bar$ & order & $\| e_0\|$ & order \\
    \hline
 1.6667e-01  & 3.3201e-01 &             &  1.6006e-02 &        \\ \hline
 8.3333e-02  & 1.6824e-01 & 9.8067e-01  &  3.9061e-03 &2.0347  \\ \hline
 4.1667e-02  & 8.4784e-02 & 9.8867e-01  &  9.6442e-04 &2.0180  \\ \hline
 2.0833e-02  & 4.2570e-02 & 9.9392e-01  &  2.3960e-04 &2.0090  \\ \hline
 1.0417e-02  & 2.1331e-02 & 9.9695e-01  &  5.9711e-05 &2.0047  \\ \hline
 5.2083e-03  & 1.0677e-02 & 9.9839e-01  &  1.4904e-05 &2.0022
  \\ \hline\hline
   \end{tabular}
\end{table}

\subsection{On Deformed Cubic Meshes}
In the fifth test, the Poisson equation is solved on a three
dimensional domain $\Omega=(0,1)^3$. The exact solution is chosen as
\begin{eqnarray*}
u=\sin(2\pi x)\sin(2\pi y)\sin(2\pi z),
\end{eqnarray*}
and the Dirichlet boundary date $g$ and $f$ are chosen accordingly
to match the exact solution.

Deformed cubic meshes are used in this test, see Figure
\ref{fig:ex4} (Left) for an illustrative element. The construction
of the deformed cubic mesh starts with a coarse mesh. The next level
of mesh is derived by refining each deformed cube element into $8$
sub-cubes, as shown in Figure \ref{fig:ex4} (Right). Table
\ref{tab:ex4} reports some numerical results for different level of
meshes. It can be seen that a convergent rate of $O(h)$ in $H^1$ and
$O(h^2)$ in $L^2$ norms are achieved for the corresponding WG finite
element solutions. This confirms the theory developed in earlier
sections.

\begin{figure}[!htb]
\centering
\begin{tabular}{cc}
  \resizebox{2.3in}{2in}{\includegraphics{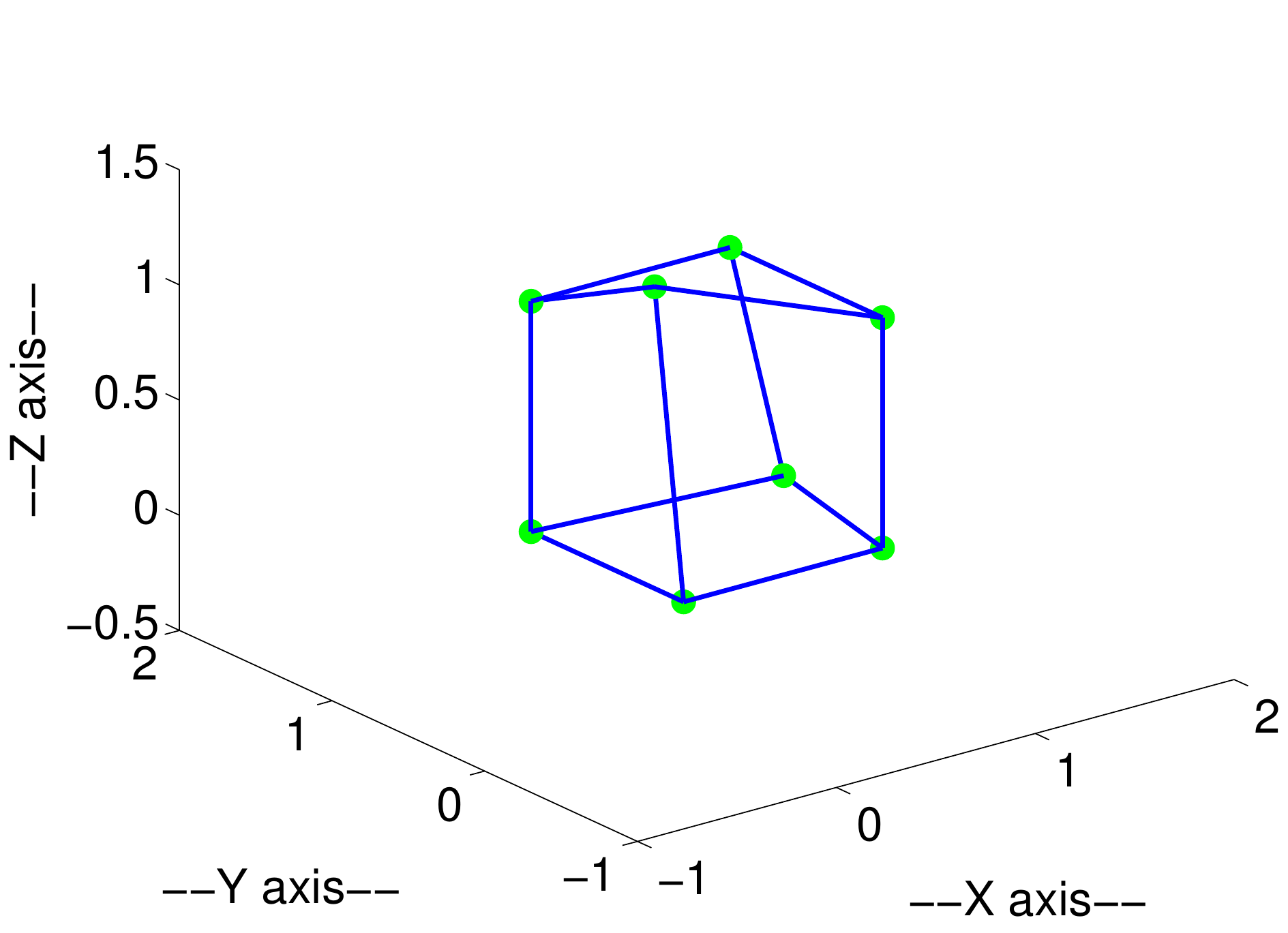}}
  \resizebox{2.4in}{2in}{\includegraphics{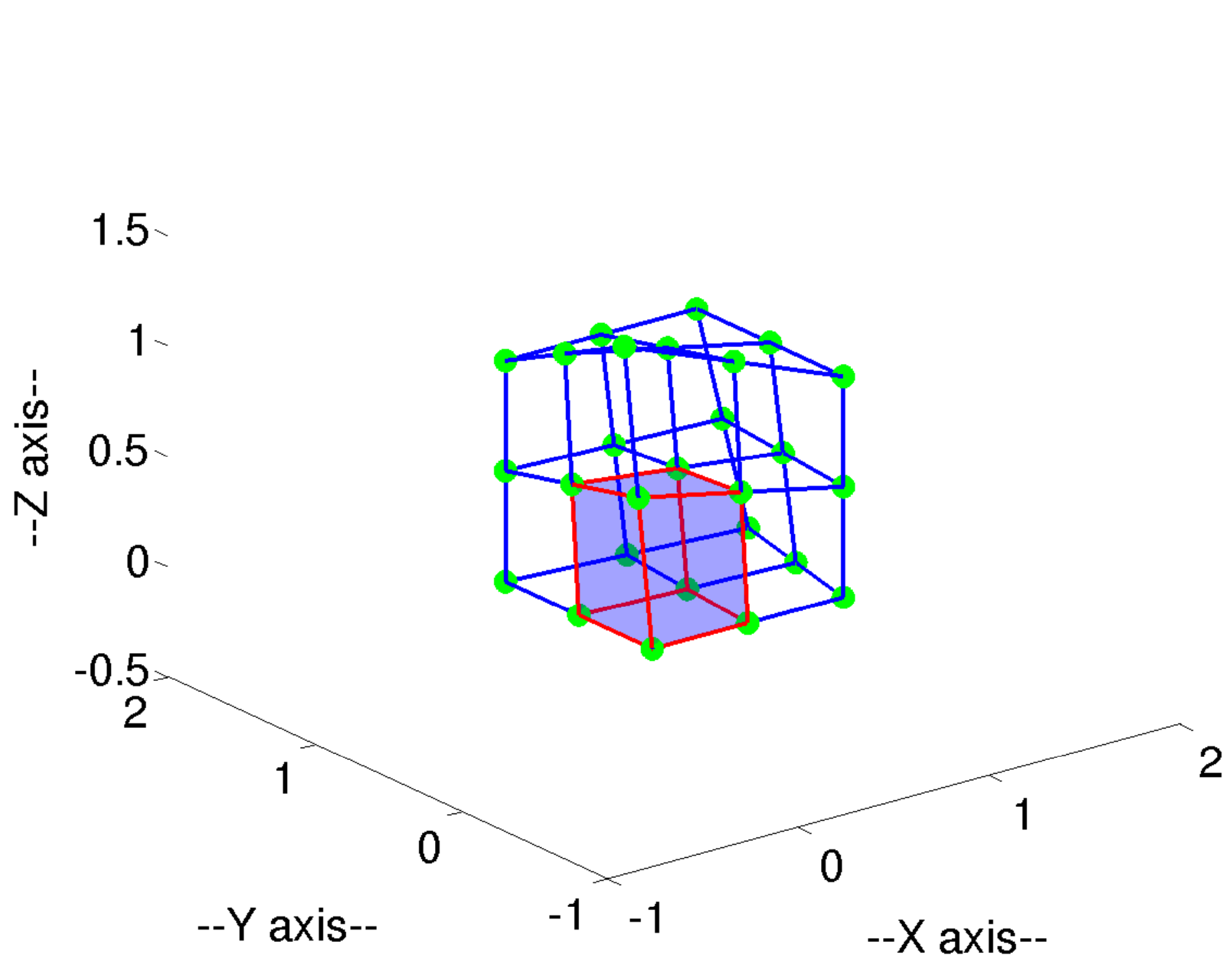}}
\end{tabular}
\caption{Mesh level 1 (Left) and mesh level 2 (Right) for example
4.}\label{fig:ex4}
\end{figure}

\begin{table}[!htb]
  \caption{Example 5. Error and convergence rate for $k=1$ on deformed cubic mesh.}
  \label{tab:ex4}
  \center
  \begin{tabular}{||c|c|c|c|c||}
    \hline\hline
   $h$ & $\3bar e_h \3bar$ & order & $\| e_0\|$ & order \\
    \hline
1/2       &5.7522     &             &9.1990     & \\ \hline
1/4       &1.3332     & 2.1092      &1.5684     & 2.5522\\ \hline
1/8       &6.4071e-01 & 1.0571      &2.7495e-01 & 2.5121\\ \hline
1/16      &3.2398e-01 & 9.8377e-01  &6.8687e-02 & 2.0011\\ \hline
1/32      &1.6201e-01 & 9.9982e-01  &1.7150e-02 & 2.0018\\ \hline\hline
   \end{tabular}
\end{table}

\vfill\eject

\end{document}